\documentclass[10pt]{amsart}
\usepackage{amssymb}
\newtheorem{lemma}{Lemma}
\newtheorem{corollary}{Corollary}
\newtheorem{theorem}{Theorem}
\newtheorem{remark}{Remark}
\newtheorem{exe}{Example}

\newtheorem{theorem a}{Theorem A}
\newtheorem{theorem b}{Theorem B}
\newtheorem{theorem C}{Theorem C}
\begin{document}

\title [New inequalities for for generalized Mathieu series and Riemann zeta functions]{Some new inequalities for Generalized  MATHIEU TYPE SERIES and Riemann zeta functions. \\}%

	\author[K. Mehrez, \v Z. Tomovski]{Khaled Mehrez and \v Zivorad Tomovski}

\address{Khaled Mehrez. D\'epartement de Math\'ematiques ISSAT Kasserine,
Universit\'e de Kairouan, Tunisia.}
\email{k.mehrez@yahoo.fr}
\address{\v{Z}ivorad Tomovski. University "St. Cyril and Methodius", Faculty
of Natural Sciences and Mathematics, Institute of Mathematics, Repubic of
Macedonia.}
\email{tomovski@pmf.ukim.edu.mk}
\maketitle

\begin{abstract} Our aim in this paper is to show some new inequalities for Mathieu's type series and Riemann zeta function. In particular, some Tur\'an type inequalities, some monotonicity and log-convexity results for these special functions are given. New Laplace type integral representations for Mathie type series and Riemann zeta function are also presented. 
\end{abstract}
\vspace{.1cm}
%%%%%%%%%%%%%%%%%%%%%%%%%%%%%%%%%%%%%%%%%%%%%%%%%%%%%%%

%%%%%%%%%%%%%%%%%%%%%%%%%%%%%%%%%%%%%%%%%%%%%

\section{Introduction}
The infinite series
\begin{equation}
S(r)=\sum_{n=1}^\infty\frac{2n}{(n^2+r^2)^2},
\end{equation}
is called a Mathieu series. It was introduced and studied by \'Emile Leonard Mathieu in his
book \cite{18} devoted to the elasticity of solid bodies.  Bounds for this series
are needed for the solution of boundary value problems for the biharmonic equations in a two--dimensional rectangular domain, see [\cite{13}, Eq. (54), p. 258].  A remarkable useful integral representation for $S(r)$ is given
by Emersleben \cite{7} in the following form
\begin{equation}
S(r)=\frac{1}{r}\int_0^\infty\frac{x\sin(rx)dx}{e^x-1}.
\end{equation}

The so-called generalized Mathieu series with a fractional power reads \cite{C}

\begin{equation}\label{0t1}
S_\mu(r)=\sum_{n=1}^\infty\frac{2n}{(n^2+r^2)^{\mu+1}},\;\mu>0,\;r>0,
\end{equation}
such series has been widely considered in mathematical literature, see \cite{C,Z1,ZK}.  Cerone and Lenard derived also the next integral expression \cite{C}
\begin{equation}\label{int}
S_\mu(r)=C_\mu(r)\int_0^\infty\frac{x^{\mu+1/2}}{e^x-1}J_{\mu-1/2}(rx)dx,\;\mu>0,
\end{equation}
where 
$$C_\mu(r)=\frac{\sqrt{\pi}}{(2r)^{\mu-1/2}\Gamma(\mu+1)}.$$

In the literature, the study of Mathieu's series and its inequalities has a rich literature,
 many interesting refinements and extensions of Mathieu's inequality can be found in \cite{ZT,ZK}.\\

In this paper is organized as follows. In section 2,  we state some useful Lemmas, which are useful in the proofs of our results. In section 3, we prove some new inequalities for Mathieu's series. In particular,
 we present the Tur\'an type inequality for this function. Moreover, we present some monotonicity and convexity results for the function $\mu\mapsto S_\mu(r).$ As consequence we establish some functional inequalities. At the end of this section, we derive the Laplace integral representation  of such series. In section 3, as applications of our main results in the section 2, we derive some new inequalities for Riemann zeta function.

Before we present the main results of this paper we recall some definitions, which will be used in the
sequel. A function $f:(0,\infty)\rightarrow\mathbb{R}$ is said to be completely monotonic if $f$ has derivatives of all orders and satisfies
$$(-1)^n f^{(n)}(x)\geq0,$$
for all $x\geq0$ and $n\in\mathbb{N}.$
 We say that a function $g:[a,b]\subseteq\mathbb{R}\rightarrow\mathbb{R}$ 
 is said to be log-convex if its natural logarithm $\log g$ is convex, that is, for all $x,y\in[a,b]$ and $\alpha\in[0,1]$ we have
$$g(\alpha x+(1-\alpha)y)\leq[g(x)]^\alpha[g(y)]^{1-\alpha}.$$

\section{Some preliminary Lemmas}

In this section, we state the following Lemmas, which are useful in the proofs of our results.

\begin{lemma}\label{l1}\cite{J}(Jensen inequality)
Let $\mu$ be a probability measure and let $\varphi\geq0$ be a convex function. Then, for all $f$ be a integrable function we have 
\begin{equation}\label{k1}
\int \varphi\circ f d\nu\geq \varphi\left(\int f d\nu\right).
\end{equation}
\end{lemma}

Our next Lemma is well-known and is stated only for easy reference, see for example [\cite{BA}. Eq. 10, p. 313]

\begin{lemma}\label{l4}Let $\mu>2.$ Then the hollowing identity 
\begin{equation}
\int_0^\infty \frac{t^{\mu-1}}{(e^t-1)^2}dt=\Gamma(\mu)\left(\zeta(\mu-1)-\zeta(\mu)\right),
\end{equation}
holds.
\end{lemma}
%\begin{proof}Let 
%$$A(\mu)=\int_0^\infty \frac{t^{\mu-1}}{(e^t-1)^2}dt,\;\;\textrm{and}\;\;f(t)=\frac{1}{e^t-1}.$$
%Thus, we have
%\begin{equation}\label{1515}
%\frac{1}{(e^t-1)^2}=-f^\prime(t)-f(t).
%\end{equation}
%Taking the Mellin transform $\mathcal{M}$ of both sides in the real variable $\mu$ in (\ref{1515}), and using the fact
%$$\mathcal{M}(f^\prime)(\mu)=-(\mu-1)\mathcal{M}(f)(\mu-1)$$
%we find 
%\begin{equation}\label{1516}
%\begin{split}
%A(\mu)&=-\mathcal{M}(f^\prime)(\mu)-\mathcal{M}(f)(\mu)\\
%&=(\mu-1)\mathcal{M}(f)(\mu-1)-\mathcal{M}(f)(\mu).
%\end{split}
%\end{equation}
%So, the fallowing integral representation 
%\begin{equation}\label{000}
%\int_0^\infty\frac{x^\mu}{e^x-1}dx=\Gamma(\mu+1)\zeta(\mu+1),
%\end{equation}
%completes the proof of Lemma \ref{l4}.
%\end{proof}

The following inequality for completely monotonic functions is due to Kimberling \cite{Ki}.

\begin{lemma} \label{l2}Let $f;(0,\infty)\longrightarrow(0,1)$ be continuous. If $f$ is completely monotonic, then
$$f(x)f(y)\leq f(x+y),\;x,y\geq0.$$
\end{lemma}

The next Lemma is given in \cite{CE}.

\begin{lemma}\label{l3} For all $\mu>0$ and $r>0.$ Then the integral of the Mathieu's series 
on $(0,\infty)$ holds true:
\begin{equation}\label{lll}
\int_0^\infty S_\mu(r)dr=\frac{\sqrt{\pi}\Gamma(\mu+1/2)\zeta(2\mu)}{\Gamma(\mu+1)}.
\end{equation}
\end{lemma}

\section{Some new inequalities for Mathieu's series}

Our main  results is the following Theorem.

\begin{theorem}Let $p>1$. Then the following inequalities
\begin{equation}\label{111}
 S_\mu^p(r)\leq\left\{ \begin{array}{ll}
 C_\mu^p(r)\Gamma^{p-1}(\mu+1/2)\zeta^{p-1}(\mu+1/2)
\Gamma(p+\mu+1/2)\zeta(p+\mu+1/2),& \textrm{if\;$\mu\geq1$}\\
\frac{C_\mu^p(r)}{2^{\frac{p}{2}}}\Gamma^{p-1}(\mu+1/2)\zeta^{p-1}(\mu+1/2)
\Gamma(p+\mu+1/2)\zeta(p+\mu+1/2)& \textrm{if\;$\mu\geq\frac{3}{2}$}
\end{array} \right.
\end{equation}
holds true for all $r\in(0,\infty),$ where $\zeta(.)$ denotes the Riemann zeta function 
defined by
$$\zeta(p)=\sum_{n=1}^\infty\frac{1}{n^p}.$$
Moreover, the following inequalities holds true 
\begin{equation}\label{k22}
S(r)\leq\frac{\sqrt{15\pi\zeta(3/2)\zeta(7/2)}}{4\sqrt{2r}}\pi,\;\textrm{and}\;\;S_{3/2}\leq\frac{\pi^3}{9\sqrt{10}r^2}.
\end{equation}
for all $r\in(0,\infty).$
\end{theorem}
\begin{proof}
Let $p>1$, we define the function $\varphi:(0,\infty)\rightarrow\mathbb{R}$ by $\varphi(x)=x^p$ 
 and we set 
$$f(x)=x,\;d\nu(x)=\frac{x^{\mu-1/2}\left|J_{\mu-1/2}(rx)\right|}{K_\mu(r)(e^x-1)}dx,$$
where 
$$K_\mu(r)=\int_0^\infty\frac{x^{\mu-1/2}\left|J_{\mu-1/2}(rx)\right|}{e^x-1}dx.$$
So, by means of Lemma \ref{l1} and the representation integral (\ref{int}) we get
\begin{equation}\label{k5}
\begin{split}
S^p(x)&\leq C_\mu^{p}(r)\left(\int_0^\infty\frac{x^{\mu+1/2}\left|J_{\mu-1/2}
(rx)\right|}{e^x-1}dx\right)^p\\
&=C_\mu^{p}(r) K_\mu^p(r)\left(\int_0^\infty x d\nu(x)\right)^p\\
&\leq C_\mu^{p}(r) K_\mu^{p-1}(r)\int_0^\infty \frac{x^{p+\mu-1/2}\left|J_{\mu-1/2}(rx)\right|}{e^x-1}dx.
\end{split}
\end{equation}
Combining the previous inequality  and von Lommel's uniform bounds [\cite{LO},\cite{24}, p. 406] 
$$\left|J_\nu(x)\right|\leq 1,\;\nu\geq0\;\;\textrm{and}\;\;
\left|J_\nu(x)\right|\leq\frac{1}{\sqrt{2}},\;\nu\geq1$$
and the representation, 
\begin{equation}\label{000}
\int_0^\infty\frac{x^\mu}{e^x-1}dx=\Gamma(\mu+1)\zeta(\mu+1),
\end{equation}
we obtain the desired inequality (\ref{111}). Finally, let $p=2,\;\mu=1,$ and $p=2,\;\mu=3/2$
 respectively in  (\ref{111}) we obtain the inequalities (\ref{k22}).
\end{proof}
\begin{theorem}Let $\mu>3/2.$ Then the following inequality 
\begin{equation}\label{3333}
S_\mu(r)\geq\frac{(2\mu-1)}{2\mu r^3}S_{\mu-1}(r)-\frac{(2\mu-1)\sqrt{\pi}\Gamma(2\mu)\zeta(2\mu-1)
}{2^{2\mu-2}r^3\Gamma(\mu+1)\Gamma(\mu+1/2)}
\end{equation}
is valid for all $r>0.$
\end{theorem}
\begin{proof}Let us consider the function $\mathcal{J}_\mu(x):\mathbb{R}\longrightarrow(-\infty,1]$
 defined by
$$\mathcal{J}_\mu(x)=\frac{2^\mu\Gamma(\mu+1) J_\mu(x)}{x^\mu},\mu>-1.$$
Thus, by (\ref{int}), we can write $S_\mu(r)$ in the following form
\begin{equation}\label{!!!.}
S_\mu(r)=c_{\mu,1}\int_0^\infty\frac{x^{2\mu}}{e^x-1}\mathcal{J}_{\mu-1/2}(rx)dx,\;\mu\geq1,
\end{equation}
where 
$$c_{\mu,1}=\frac{\sqrt{\pi}}{2^{2\mu-1}\Gamma(\mu+1/2)\Gamma(\mu+1)}.$$
 By using the differentiation formula [\cite{24}, p.18]
$$\mathcal{J}_\mu^\prime(x)=-\frac{x}{2(\mu+1)}\mathcal{J}_{\mu+1}(x),$$
and the integrating by parts in the right hand side of (\ref{!!!.}), we get 
\begin{equation}\label{!}
\begin{split}
S_\mu(r)&=-\frac{(2\mu-1)c_{\mu,1}}{r^2}\int_0^\infty\frac{x^{2\mu-1}}{e^x-1}
\mathcal{J}_{\mu-3/2}^\prime(rx)dx\\
&=\frac{(2\mu-1)c_{\mu,1}}{r^2}\left[\int_0^\infty\frac{(2\mu-1)x^{2\mu-2}}{r(e^x-1)}
\mathcal{J}_{\mu-3/2}(rx)dx-\int_0^\infty\frac{x^{2\mu-1}e^x}{r(e^x-1)^2}
\mathcal{J}_{\mu-3/2}(rx)dx\right]\\
&=\frac{(2\mu-1)^2c_{\mu,1}}{c_{\mu-1,1}r^3} S_{\mu-1}(r)-\frac{(2\mu-1)c_{\mu,1}}{r^3}
\left[\int_0^\infty\frac{x^{2\mu-1}}{e^x-1}
\mathcal{J}_{\mu-3/2}(rx)dx+\int_0^\infty\frac{x^{2\mu-1}}{(e^x-1)^2}
\mathcal{J}_{\mu-3/2}(rx)dx\right]\\
&\geq\frac{(2\mu-1)^2c_{\mu,1}}{c_{\mu-1,1}r^3} S_{\mu-1}(r)-\frac{(2\mu-1)c_{\mu,1}}{r^3}
\left[\int_0^\infty\frac{x^{2\mu-1}}{e^x-1} dx+\int_0^\infty\frac{x^{2\mu-1}}{(e^x-1)^2}
dx\right]\\
&=\frac{(2\mu-1)^2c_{\mu,1}}{c_{\mu-1,1}r^3} S_{\mu-1}(r)-\frac{(2\mu-1)c_{\mu,1}}{r^3}
\left[\Gamma(2\mu)\zeta(2\mu)+\Gamma(2\mu)(\zeta(2\mu-1)-\zeta(2\mu))\right]\\
&=\frac{(2\mu-1)^2c_{\mu,1}}{c_{\mu-1,1}r^3} S_{\mu-1}(r)-\frac{(2\mu-1)c_{\mu,1}
\Gamma(2\mu)\zeta(2\mu-1)}{r^3}.
\end{split}
\end{equation} 
In the equation (\ref{!}), we use the bound by Minakshisundaram and Sz\'asz \cite{MSZ}
$$|\mathcal{J}_{\mu}(x)|\leq 1,\;\mu>\frac{-1}{2},\;x\in\mathbb{R},$$
and Lemma \ref{l4}. The desired inequality (\ref{3333}) is established.
\end{proof}

In the next Theorem we establish the Tur\'an type inequalities for the Mathieu's serie $S_\mu(r).$

\begin{theorem}\label{T2}Let $\mu>0.$ Then the Tur\'an type inequality 
\begin{equation}\label{Turan}
 S_{\mu+2}(r)S_\mu(r)-S_{\mu+1}^2(r)\geq0,
\end{equation}
holds true for all $r\in(0,\infty).$
\end{theorem} 
\begin{proof}
The Cauchy product reveals
$$S_{\mu+2}(r)S_\mu(r)-S_{\mu+1}^2(r)=$$
\begin{equation}
\begin{split}
 &\;\;\;\;\;\;\;\;\;\;\;\;\;\;\;\;\;\;\;\;\;\;\;\;\;\;\;\;\;\;\;\;=\sum_{n=1}^\infty\sum_{k=1}^n\frac{4k(n-k)}
{(k^2+r^2)^{\mu+1}((n-k)^2+r^2)^{\mu+3}}-\sum_{n=1}^\infty\sum_{k=1}^n\frac{4k(n-k)}
{(k^2+r^2)^{\mu+2}((n-k)^2+r^2)^{\mu+2}}\\
&\;\;\;\;\;\;\;\;\;\;\;\;\;\;\;\;\;\;\;\;\;\;\;\;\;\;\;\;\;\;\;\;=\sum_{n=1}^\infty\sum_{k=1}^n\frac{4nk(n-k)(2k-n)}
{(k^2+r^2)^{\mu+2}((n-k)^2+r^2)^{\mu+3}}\\
&\;\;\;\;\;\;\;\;\;\;\;\;\;\;\;\;\;\;\;\;\;\;\;\;\;\;\;\;\;\;\;\;=\sum_{n=1}^\infty\sum_{k=0}^n 4n(2k-n)T_{n,k},
\end{split}
\end{equation}
where
$$T_{n,k}=\frac{k(n-k)}{((n-k)^2+r^2)^{\mu+3}(k^2+r^2)^{\mu+2}}.$$
 If $n$ is even, then 
\begin{equation}
\begin{split}
\sum_{k=0}^n T_{n,k}(2k-n)&=\sum_{k=0}^{n/2-1} T_{n,k}(2k-n)+
\sum_{k=n/2+1}^n T_{n,k}(2k-n)\\
&=\sum_{k=0}^{[(n-1)/2]} (T_{n,k}-T_{n,n-k})(2k-n),
\end{split}
\end{equation}
where $[.]$ denotes the greatest integer function. Similarly, if $n$ is odd, then
$$\sum_{k=0}^n T_{n,k}(2k-n)=\sum_{k=0}^{[(n-1)/2]} (T_{n,k}-T_{n,n-k})(2k-n).$$
Thus,
\begin{equation}\label{KH48}
 S_{\mu+2}(r)S_\mu(r)-S_{\mu+1}^2(r)=\sum_{n=1}^\infty\sum_{k=0}^{[(n-1)/2]} 
(T_{n,k}-T_{n,n-k})(2k-n).
\end{equation}
Simplifying, we find that
\begin{equation}\label{KH47}
T_{n,k}-T_{n,n-k}=\frac{n^2k(n-k)(2k-n)}{((n-k)^2+r^2)^{\mu+3}(k^2+r^2)^{\mu+3}}.
\end{equation}
In view of (\ref{KH48}) and (\ref{KH47}), we deduce that the Tur\'an type inequality (\ref{Turan}) holds.
\end{proof}

\begin{theorem}The following assertion are true:\\
\noindent 1. The function $\mu\mapsto S_\mu(r)$ is completely monotonic and log-convex on 
$(0,\infty)$ for each $r>0.$\\
\noindent 2. The function $\mu\mapsto \frac{S_{\mu+1}(r)}{S_{\mu}(r)}$ is increasing on $(0,\infty).$\\
\noindent 3. Furthermore, for all $r>0$, the following inequalities are valid
\begin{equation}\label{TTTT1}
S_{\mu+\nu}(r)S(r)\geq S_\mu(r)S_\nu(r),\;\mu,\nu>0.
\end{equation}
\begin{equation}\label{TTTT2}
\left[\frac{S_\nu(r)}{\zeta(2\nu+1)}\right]^{\frac{1}{\nu+1}}\geq
 \left[\frac{S_\mu(r)}{\zeta(2\mu+1)}\right]^{\frac{1}{\mu+1}},\;\mu\geq\nu>0.
\end{equation}
\begin{equation}\label{TTTT3}
\left[\frac{S_\mu(r)}{\zeta(2\mu+1)}\right]^{\frac{1}{\mu+1}}+
\frac{\zeta(2\mu+3)S_\mu(r)}{\zeta(2\mu+1)S_{\mu+1}(r)}\geq2,\mu>0.
\end{equation}
for all $\mu,\nu>0.$
\end{theorem}
\begin{proof}1. For all $m\in\mathbb{N},\;\mu>0$, we have
$$(-1)^m\frac{\partial^m S_\mu(r)}{\partial\mu^m}=\sum_{n\geq1}\frac{2n\log^m(n^2+r^2)}
{(n^2+r^2)^{\mu+1}}\geq0.$$
Thus, the function $\mu\mapsto S_\mu(r)$ is completely monotonic and Log-convex on $(0,\infty)$,
 since every completely monotonic function is log-convex, see [\cite{WI}, p. 167].\\
2. From part 1. of this Theorem, the function $\mu\mapsto\log S_\mu(r)$ is convex and hence, it follows that the function   
$\mu\mapsto \log S_{\mu+1}(r)-\log S_\mu(r)$ is increasing.\\
3. Since the function $\mu\mapsto\frac{S_\mu(r)}{S(r)}$ is completely monotonic 
on $(0,\infty)$ and maps $(0,\infty)$ to $(0,1)$, according to Lemma \ref{l2}, 
we conclude the asserted inequality (\ref{TTTT1}). Newt, we prove the inequality (\ref{TTTT2}).
 Suppose that $\mu\geq\nu>0$ and define the function $H:(0,\infty)\longrightarrow\mathbb{R}$ with relation
$$H(r)=\frac{\nu+1}{\mu+1} \log S_\mu(r)-\log S_\nu(r).$$
On the other hand, by using the fact
$$S_\mu^\prime(r)=-2r(\mu+1)S_{\mu+1}(r),$$
we have
$$H^\prime(r)=2(\nu+1)r\left(\frac{S_{\nu+1}(r)}{S_\nu(r)}-\frac{S_{\mu+1}(r)}
{S_\mu(r)}\right)\leq0.$$
So, by part 2. in this Theorem we conclude that the function $H(r)$ 
is decreasing on $(0,\infty)$. Consequently, $H(r)\leq H(0)$. Replacing $\mu$ by $\mu+1$
and $\nu$ by $\mu$ in inequality  (\ref{TTTT3}), we use  the inequality (\ref{TTTT2})  and the arithmetic--geometric mean inequality
$$\frac{1}{2}\left(\left[\frac{S_\mu(r)}{\zeta(2\mu+1)}\right]^{\frac{1}{\mu+1}}+
\frac{\zeta(2\mu+3)S_\mu(r)}{\zeta(2\mu+1)S_{\mu+1}(r)}\right)\geq \sqrt{\left[\frac{S_\mu(r)}{\zeta(2\mu+1)}\right]^{\frac{1}{\mu+1}}.
\frac{\zeta(2\mu+3)S_\mu(r)}{\zeta(2\mu+1)S_{\mu+1}(r)}}\geq1.$$
\end{proof}
\begin{remark}We note  that there are another proofs of the Tur\'an type inequality (\ref{Turan}). Indeed, since the function $\nu\mapsto S_\mu(r)$ is log-convex on $(0,\infty)$ for $r>0,$ it follows that for all $\mu_1,\mu_2\geq1,\;
\alpha\in[0,1]$ and $r>0$ we have
$$S_{\alpha\mu_1+(1-\alpha)\mu_2}(r)\leq \left[S_{\mu_1}(r)\right]^\alpha
\left[S_{\mu_2}(r)\right]^{1-\alpha}.$$
Choosing $\mu_1=\mu,\:\mu_2=\mu+2$ and $\alpha=\frac{1}{2}$, the above inequality
 reduces to the Tur\'an inequality (\ref{Turan}). A third proof of this inequality
 can be obtained as follows. By using the fact that the function
 $\mu\mapsto\frac{S_{\mu+1}(r)}{S_\mu(r)}$ is increasing on $(0,\infty)$, we have
$$\frac{S_{\mu+2}(r)}{S_{\mu+1}(r)}\geq\frac{S_{\mu+1}(r)}{S_\mu(r)},$$
and hence the required result follows.
\end{remark}
\begin{theorem}\label{t4} For $\mu\geq1.$ The Mathieu series admits the following integral 
representation
\begin{equation}\label{333}
S_\mu(r)=c_\mu\int_0^\infty e^{-rt} K_\mu(t)dt
\end{equation}
and
\begin{equation}\label{33333}
\zeta(2\mu+1)=\frac{c_\mu}{2}\int_0^\infty K_\mu(t)dt.
\end{equation}
and,
\begin{equation}\label{3333}
S(r)=\int_0^\infty e^{-rt} K(t)dt
\end{equation}
where $c_\mu =\frac{\sqrt{\pi}}{2^{\mu-1/2}\Gamma(\mu+1)}$ and $K_\mu(t)$ and $K(t)$ 
are defined by
$$K_\mu(t)=t^{\mu+1/2}g_\mu(t)\;\textrm{and}\;K(t)=h(t)-th^\prime(t),$$
with $g_\mu$ and $h(t)$ are Kapteyn series, defined as
$$g_\mu(t)=\sum_{n=1}^\infty \frac{J_{\mu+1/2}(nt)}{n^{\mu-1/2}},\;\textrm{and}\;h(t)=\sum_{n=1}^\infty\frac{\sin(nt)}{n^2}$$
where $J_\mu(.)$ is the Bessel function.
\end{theorem}
\begin{proof}By using the formula [\cite{DE}, eq. 42, p. 397]
\begin{equation}
\frac{1}{(a^2+s^2)^{\mu+1/2}}=\frac{\sqrt{\pi}}{(2a)^{\mu+1/2}\Gamma(\mu+1/2)}
\int_0^\infty t^\mu e^{-st} J_\mu(at)dt,\;\mu>-1/2
\end{equation}
we get
\begin{equation}
\frac{2n}{(n^2+r^2)^{\mu+1}}=\frac{\sqrt{\pi}}{(2n)^{\mu-1/2}\Gamma(\mu+1)}
\int_0^\infty t^{\mu+1/2} e^{-rt}
J_{\mu+1/2}(nt)dt.
\end{equation}
The interchanging between integral and summation gives (\ref{333}). Now, 
let $r$ tends to $0$ in (\ref{333}) we obtain that  (\ref{33333}) holds true.  Finally, let $\mu=1$ in  (\ref{333}) and using the fact 
$$J_{3/2}(t)=\sqrt{\frac{2}{\pi}}\left(\frac{\sin x}{x^{3/2}}-\frac{\cos x}{x^{1/2}}\right),$$
we obtain
$$S(r)=\int_0^\infty e^{-rt}\left(\sum_{n=1}^\infty\frac{\sin (nt)}{n^2}-t
\sum_{n=1}^\infty\frac{\cos(nt)}{n}\right)dt.$$
 The proof of Theorem \ref{t4} is complete.
\end{proof}
%\begin{corollary} For $\mu>1,$ the functions $K_\mu$ and $K$ admits the following integral representation
%\begin{equation}\label{KKK1}
%K_\mu(t)=\frac{2^{\mu-1/2}\Gamma(\mu+1)}{i\pi\sqrt{pi}\Gamma(2\mu+1)}\int_{\gamma-i\infty}^{\gamma+i\infty}\int_0^\infty
%\frac{x^{2\mu}e^{rt}}{e^x-1}\;{}_2F_1\left(1;1/2;\mu+1/2;-\frac{x^2}{r^2}\right)dx dr,
%\end{equation}
%and
%\begin{equation}\label{KKK2}
%K(t)=\frac{1}{2i\pi\Gamma(2\mu+1)}\int_{\gamma-i\infty}^{\gamma+i\infty}
%\int_0^\infty\frac{x e^{rt}}{e^x-1}\arctan\left(\frac{x}{r}\right)dx dr.
%\end{equation}
%\end{corollary}
%\begin{proof} In \cite{Z3}, the authors proved the following integral representation of 
%the functions $S(r)$ and $S_\mu(r)$ as
%\begin{equation}\label{yy}
%\mathcal{L}(S_\mu(r))(x)=\frac{2}{x\Gamma(2\mu+1)}\int_0^\infty \frac{t^{2\mu}}{e^t-1}\;{}_2F_1\left(1;1/2;\mu+1/2;-\frac{x^2}{r^2}\right)dt.
%\end{equation}
%and
%\begin{equation}\label{yyy}
%\mathcal{L}(S(r))(x)=\int_0^\infty \frac{t}{e^t-1}\arctan\left(\frac{t}{x}\right)dt,
%\end{equation}
%Combining (\ref{333}) and (\ref{yy}) (respectively (\ref{3333}) and (\ref{yyy})) and using the definition of inverse Laplace transform we deduce the integral representation (\ref{KKK1}) ( respectively  (\ref{KKK2})). 
%\end{proof}
\begin{exe} Let $r$ tends to $0$ in (\ref{3333})
 we get integral formula for the Apery constant
\begin{equation}
\zeta(3)=\frac{1}{2}\int_0^\infty K(t) dt=\frac{1}{2}\int_0^\infty\left(\sum_{n=1}^\infty\frac{\sin (nt)}{n^2}-t
\sum_{n=1}^\infty\frac{\cos(nt)}{n}\right)dt.
\end{equation}
\end{exe} 
\begin{corollary}Let $\mu>7/6.$ Then the following inequality 
\begin{equation}\label{TTT}
S_\mu(r)\leq \frac{c_L\sqrt{\pi}\Gamma(\mu+7/6)\zeta(\mu-1/6)}{2^{\mu-1/2}\Gamma(\mu+1)
r^{\mu+7/6}},
\end{equation}
is valid for all $r>0,$ where $c_L=\sup_{x>0}\{x^{1/3}J_0(x)\}= 0.78574687...$
\end{corollary}
\begin{proof}By using the Landeau estimate (see \cite{LA})
$$|J_\mu(x)| \leq c_L x^{-1/3},$$ where $c_L=\sup_{x>0}\{x^{1/3}J_0(x)\}= 0.78574687...$  
uniformly in $\mu,$ and using the integral representation (\ref{333}),
 we deduce that the inequality (\ref{TTT}) holds true.
\end{proof}
\section{some new inequalities for Riemann Zeta Functions}

Firstly results in this section, we present new Tur\'an type inequality for Riemann Zeta functions.

\begin{theorem}Let $\mu>1.$ Then the Tur\'an type inequality
\begin{equation}\label{T789}
\zeta(\mu)\zeta(\mu+2)-\zeta^2(\mu+1)\geq0.
\end{equation}
holds true.
\end{theorem}
\begin{proof} Letting $r$ tends to $0$ in (\ref{Turan}), we get 
$$\zeta(2\mu+1)\zeta(2\mu+3)-\zeta^2(2\mu+2)\geq0.$$
Replacing $2\mu+1$ by $\mu$ in the previous inequality we find that the inequality (\ref{T789}) is valid.
\end{proof}
\begin{remark} In \cite{LF}, Laforgia and Natalini by using the  generalization of the Schwarz inequality proved the following Tur\'an type inequality for Riemann zeta function 
\begin{equation}\label{lafo}
\zeta(\mu)\zeta(\mu+2)\geq\frac{\mu}{\mu+1}\zeta^2(\mu+1),\;\mu>1.
\end{equation}
We note that the inequality (\ref{T789}) is better than the inequality (\ref{lafo}).
\end{remark}

In the next Theorem we establish a simple upper bounds of the Zeta function. Our main tool 
will be the formula (\ref{lll}) in Lemma \ref{l3}.

\begin{theorem}Let $\mu\geq1.$ Then the following inequalities holds true,
\begin{equation}\label{147}
\zeta(2\mu)\leq\sqrt{\frac{3\pi}{2}}\frac{\Gamma(\mu+1)}{\Gamma(\mu+1/2)}
\end{equation}
and 
\begin{equation}\label{zer}
\frac{\zeta^{\frac{3}{2}}(2\mu+1)}{\zeta^2(2\mu)\zeta(2\mu+3)}\leq\frac{(\mu+1)\Gamma^2(\mu+1/2)}{\Gamma^2(\mu+1)}
\end{equation}
\end{theorem}
\begin{proof}In \cite{Al}, the authors proved that 
\begin{equation}
S(r)<\frac{1}{r^2+1/6},\; r>0,
\end{equation}
and since the function $\mu\mapsto S_\mu(r)$ is decreasing on $[1,\infty)$, we get 
\begin{equation}\label{zzz}
S_\mu(r)<\frac{1}{r^2+1/6},\; r>0,
\end{equation}
for all $\mu\geq1.$ Integrating (\ref{zzz}) on the interval $(0,\infty)$, we obtain 
$$\int_0^\infty S_\mu(r)dr\leq\sqrt{\frac{3}{2}}\pi.$$ 
So, Lemma \ref{l3} completes the proof of inequality (\ref{147}). Now, we proved the inequality 
(\ref{zer}). In \cite{ZK}, the authors of this paper obtained that 
\begin{equation}\label{re1}
2\zeta(2\mu+1)\exp\left\{-(\mu+1)\frac{\zeta(2\mu+3)}{\zeta(2\mu+1)}r^2\right\}\leq S_\mu(r),\;r>0.
\end{equation}
Therefore, integrating (\ref{re1}) and from the Lemma \ref{l1}, we deduce that 
the inequality (\ref{zer}) holds.

\end{proof}

\end{document}